\documentclass[a4paper,12pt]{amsart}
\usepackage[left=2cm,top=2.5cm,right=2cm,bottom=2.5cm]{geometry}
\usepackage{amsmath}
\usepackage{amssymb}
\usepackage{amsfonts}
\usepackage{amsthm}
\usepackage{tikz}

\usepackage{stmaryrd}

\newtheorem{thm}{Theorem}[section]
\newtheorem{lemma}[thm]{Lemma}
\newtheorem{prop}[thm]{Proposition}

\theoremstyle{definition}

\newtheorem*{rem}{Remark}
\newtheorem*{qu}{Question}

\theoremstyle{definition}
\newtheorem{defn}[thm]{Definition}

\DeclareMathOperator{\vol}{vol}

\DeclareMathOperator{\inter}{int}
\DeclareMathOperator{\conv}{conv}

\newcommand{\norm}[1]{{\left\|{#1}\right\|}}

\newcommand{\abs}[1]{{\left|{#1}\right|}}
\newcommand{\scal}[1]{{\left\langle{#1}\right\rangle}}
\newcommand{\set}[1]{{\left\{{#1}\right\}}}

\newcommand{\eps}{\ensuremath{\varepsilon}}

\newcommand{\ka}{\ensuremath{\kappa}}
\newcommand{\La}{\ensuremath{\Lambda}}
\newcommand{\la}{\ensuremath{\lambda}}
\newcommand{\sig}{\ensuremath{\sigma}}
\newcommand{\vphi}{\ensuremath{\varphi}}
\newcommand{\om}{\ensuremath{\omega}}
\newcommand{\Om}{\ensuremath{\Omega}}

\newcommand{\bears}{\begin{eqnarray*}}
\newcommand{\eears}{\end{eqnarray*}}

\newcommand{\mc}[1]{\ensuremath{\mathcal{#1}}}

\newcommand{\ZZ}{\ensuremath{\mathbb{Z}}}

\newcommand{\QQ}{\ensuremath{\mathbb{Q}}}
\newcommand{\RR}{\ensuremath{\mathbb{R}}}
\newcommand{\NN}{\ensuremath{\mathbb{N}}}
\newcommand{\TT}{\ensuremath{\mathbb{T}}}

\newcommand{\sm}{\ensuremath{\setminus}}

\newcommand{\rar}{\ensuremath{\rightarrow}}

\newcommand{\bs}[1]{\ensuremath{\mathbf{#1}}}
\providecommand{\abs}[1]{\lvert#1\rvert}

\numberwithin{equation}{section}

\title{Polyhedral Gauss Sums, and polytopes with symmetry}
\makeatletter
\@namedef{subjclassname@2010}{\textup{2010} Mathematics Subject Classification}
\makeatother

\date{\today}
\author{Romanos-Diogenes Malikiosis} 
\thanks{R.~D. Malikiosis is supported with a Postdoctoral Fellowship from Humboldt Foundation.}
\address{Technische Universit\"at Berlin, Institut f\"ur Mathematik,
Sekretariat MA 4-1,
Stra{\ss}e des 17. Juni 136,
D-10623 Berlin, Germany}
\email{malikios@math.tu-berlin.de}

\author{Sinai Robins}
\thanks{S. Robins was supported by the Singapore Ministry of Education ARF Tier 2 Grant MOE2011-T2-1-090, and by ICERM, The Institute for Computational and Experimental  Research in Mathematics}
\address{Department of Mathematics,
Brown University,
Box 1917, 151 Thayer Street,
Providence, RI 02912}
\email{sinai\textunderscore\hspace{.05cm}robins@brown.edu}

\author{Zhang Yichi}
\address{Division of Mathematical Sciences,  
Nanyang Technological University,
SPMS-MAS-03-01, 
\quad 21 Nanyang Link,
Singapore 637371} 
\email{YCZhang@ntu.edu.sg}

\date{}

\begin{document}

\begin{abstract}
We define certain natural finite sums of $n$'th roots of unity, called $G_P(n)$, that are associated to each convex integer polytope $P$, and which generalize the classical $1$-dimensional Gauss sum $G(n)$ defined over $\mathbb Z/ {n \mathbb Z}$, to higher dimensional abelian groups and integer polytopes.   We consider the finite Weyl group $\mc{W}$, generated by the reflections with respect to the coordinate hyperplanes, as well as all permutations of the coordinates;  further, we let $\mc G$  be the group generated by $\mc{W}$ as well as all integer translations in $\mathbb Z^d$.   We prove that if $P$ multi-tiles $\mathbb R^d$ under the action of $\mc G$, then we have the closed form $G_P(n) = \vol(P) G(n)^d$.   Conversely, we also prove that if $P$ is a lattice tetrahedron in $\mathbb R^3$, of volume $1/6$, such that $G_P(n) = \vol(P) G(n)^d$, for $n \in \{ 1,2,3,4 \}$, then there is an element $g$ in $\mc G$ such that $g(P)$ is the fundamental tetrahedron with vertices $(0,0,0)$, $(1, 0, 0)$, $(1,1,0)$, $(1,1,1)$.
\end{abstract}
\maketitle

\bigskip
\bigskip

\section{Introduction}

Our goal is to define certain finite sums of roots of unity, associated to a convex lattice polytope $P$,   in order to help us determine whether $P$ has certain symmetries and in fact whether $P$ is a fundamental domain of a certain Weyl group.   For $3$-dimensional integer tetrahedra $P$, we discover that certain natural generalizations of the classical $1$-dimensional Gauss sums, which we call polyhedral Gauss sums, collapse to a closed form over $P$ if and only if $P$ is a fundamental domain of a Weyl group.  

Intuitively, we are projecting the structure of $P$ onto the $2$-dimensional complex plane, and seeing what a closed form of its associated Gauss sum of roots of unity in the complex plane tells us about the question of whether or not P is a fundamental domain for some group acting on P.   It is much easier to handle $2$-dimemsional computations directly than $d$-dimensional geometric computations, and surprisingly we can discern the geometry of $P$ in a very detailed way by sufficiently many of these computations with roots of unity.   From a number-theoretic perspective, these computations generalize the classical $1$-dimensional results of Gauss to $d$-dimensional integer polytopes. 

 Gauss sums over finite abelian groups have been studied by \cite{Siegel} and \cite{Krazer, Turaev}, and they can be viewed as the study of Gauss sums over integer parallelepipeds, because when we quotient $\mathbb Z^d$ by the discrete subgroup generated by the edge vectors of an integer parallelepiped, we get a finite abelian group.  Here we extend the closed form results in the existing literature on Gauss sums over parallelepipeds, to more general Gauss sums over integer polytopes.  
 
 In one direction, if we assume that $P$ is any $d$-dimensional  integer polytope that tiles or multi-tiles Euclidean space by a Weyl group, then we can show that its corresponding polyhedral Gauss sum always achieves a nice closed form, proportional to the volume of $P$.   In the other direction, for $d=3$,  if we assume that the polyhedral Gauss sum of certain integer tetrahedra $P$ achieve a closed form proportional to their volume, then we show $P$ must be a fundamental domain for a certain Weyl group. 

In order to precisely define our generalized Gauss sums, we first need the notion of a solid angle at any point $x \in \mathbb R^d$, relative to a fixed polytope $P$.   We let $1_P$ be the indicator function of $P$, and we define the solid angle at any point $x \in \mathbb R^d$ by 

 \begin{equation}
 \omega_P(x) :=  \frac{   \vol(B(x,r) \cap P) }{\vol(B(x,r))},
 \end{equation}
 for all sufficiently small values of $r >0$.
 Some obvious but noteworthy properties of $\omega_P$ are the following:
$\omega_P(x)=1$ if $x\in\inter{P}$ and $\omega_P(x)=0$ if $x\notin P$.  For the non-trivial case that 
$x\in\partial P$ (the boundary of $P$),  $\omega_P(x)$ is equal to the solid angle of the smallest cone containing $P$ with
  apex at $x$.

 \begin{defn}
  The polyhedral Gauss sum over $P$ is defined by
  \[G_P(n)=\sum_{x\in\ZZ^d}\omega_{nP}(x)e\left(\frac{\norm{x}^2}{n}\right),\]
  for $n\in\NN$, where $nP$ denotes the dilation of $P$ by $n$, and as usual, $e(x):=e^{2\pi ix}$.
 \end{defn}

The classical $1$-dimensional Gauss sum, for example, is the case of the $1$-dimensional polytope $P=[0,1]$, and for this important case we 
define 
 \[
 G(n)=\sum_{k\in\ZZ/n\ZZ}e\left(\frac{k^2}{n}\right).
 \]
 
 Gauss discovered a closed form for this $1$-dimensional Gauss sum \cite{NZM}, given by:   
 
\begin{equation}
G(n):=\sum_{k=0}^{n-1} e\left(\frac{2\pi ik^2}{n}\right)=\left\{
\begin{array}{lr}
(1+i)\sqrt{n} & n \equiv 0\;\;mod\;\;4\\
\sqrt{n} & n\equiv 1\;\;mod\;\;4\\
0 & n\equiv 2\;\;mod\;\;4\\
i\sqrt{n} & n\equiv 3\;\;mod\;\;4
\end{array}
\right.
\end{equation}

 It is natural to wonder what geometric properties an integer polytope must possess in order to achieve similar closed forms in higher dimensions.   To this end we have the following result. 
 
 \begin{thm}\label{main1}
  If $P$ multi-tiles the space $\mathbb R^d$ with the group $\mc{G}$, then
  \[G_P(n)=\vol(P)G(n)^d.\]
 \end{thm}

In general, the converse question of whether such a closed form for a polyhedral Gauss sum over an integer polytope $P$ implies that $P$ must tile or multi-tile Eucliden space 
seems to be out of reach for general polytopes in dimension $d\geq3$.  However, we discovered a partial converse for $d=3$ and in the case that $P$ belongs to a class of integer simplices.

\begin{thm}\label{converse}
 Let $T$ be a lattice tetrahedron of volume $1/6$, such that $G_T(n)=\vol(T)G(n)^3$ for $n \in \{1,2,3,4\}$. Then there is an element $g$ in the Weyl group $\mc{W}$ such that $g(T)$ 
 is the tetrahedron with vertices $(0,0,0),(1,0,0),(1,1,0),(1,1,1)$.
\end{thm}

\bigskip
\bigskip
 \section{Preliminaries}
 
 The Weyl group   is the finite group generated by reflections with respect to the coordinate hyperplanes, as well
 as permutations of coordinates. We denote it by $\mc{W}$, and its cardinality is $2^d d!$. In this note, we
 will deal with sets that multi-tile the space under the action of $\mc{G}$, the group of operators generated by
 $\mc{W}$ and all lattice translations. Clearly, $\mc{G}\cong \mc{W}\times\ZZ^d$.
 
 The orbit of any point $x$ under the action of $\mc{G}$ is denoted by $\mc{G}(x)$,  and the stabilizer of any $x$ is denoted by
 $\mc{G}_x$ (and similarly for $\mc{W}$). Obviously, $\mc{G}_x$ is finite for all $x$, as $x$ cannot remain invariant
 under any lattice translation, and almost all $x$ have full orbit, i.e. $\abs{\mc{G}_x}=1$ except for a set of
 Lebesgue measure zero. Furthermore, the action of $\mc{W}$ can be restricted to $[0,1)^d=\TT^d$, a fundamental
 domain  for the action of the group $\ZZ^d$, acting by translations on $\mathbb R^d$; usually we will treat elements of $\TT^d$ as elements of $\RR^d$. Then, it is not hard
 to verify that $\abs{\mc{G}_x}=\abs{\mc{W}_x}$.
 
 There are many choices of fundamental domains for $\mc{G}$, and a natural choice for such a fundamental domain is the tetrahedron 
 \[
 T=\set{(x_1,\dotsc,x_d)\in\RR^d\vert 0\leq x_1\leq\dotsb\leq x_d\leq1/2},
 \]
 which is also a fundamental domain of $\mc{W}$ acting on $\TT^d$.

 \begin{defn}\label{multitile}
 We say that $P$ \emph{multitiles} $\mathbb R^d$, with multiplicity  $m$, if 
  $\sum_{g\in\mc{G}}1_P(gx)=m$ for almost all $x \in \mathbb R^d$.
  \end{defn}

Equivalently, we may also say that $P$ multi-tiles with multiplicity $m$ if  $\abs{\mc{G}(x)\cap P}=m$, for almost all $x$.
It is clear from definition \ref{multitile}  that this $m$ must be a positive integer.
 Next, define the functions $f_P$ and $g_P$ on $T$ as follows:
 \[f_P(x)=\sum_{g\in\mc{G}}\omega_P(gx),\	g_P(x)=\sum_{y\in\mc{G}(x)}\omega_P(y).\]
 Obviously, $f_P=g_P$ almost everywhere; in particular
 \[g_P(x)=\frac{1}{\abs{\mc{W}_x}}f_P(x),\]
 so they differ only on the boundary of $T$.
 
 \bigskip
 \begin{prop}\label{const}
  If $P$ multi-tiles the space, then $f_P$ is constant, equal to $\abs{\mc{W}}\vol(P)$.
 \end{prop}
 
 \begin{proof}
  By definition, $\abs{\mc{G}(x)\cap P}=m$ for almost all $x$ and some positive integer $m$.   Then, for \emph{all} $x\in T$ we have
  \begin{eqnarray*}
   f_P(x)=\sum_{g\in\mc{G}}\omega_P(gx) &=& 
   \sum_{g\in\mc{G}}\lim_{r\rar0}\frac{1}{\vol(B(gx,r))}\int_{B(gx,r)}1_P(y)dy\\
   &=& \lim_{r\rar0}\frac{1}{\vol(B(x,r))}\sum_{g\in\mc{G}}\int_{B(x,r)}1_P(gy)dy\\
   &=& \lim_{r\rar0}\frac{1}{\vol(B(x,r))}\int_{B(x,r)}\sum_{g\in\mc{G}}1_P(gy)dy\\
   &=& m.
  \end{eqnarray*}
The above sum commutes with the limit and the integral, because it is finite. For the second part,
\begin{eqnarray*}
 \frac{m}{\abs{\mc{W}}}=\int_T f(x)dx &=& \int_T\sum_{g\in\mc{G}}\omega_P(gx)dx\\
 &=& \sum_{g\in\mc{G}}\int_T \omega_P(gx)dx\\
 &=& \sum_{g\in\mc{G}}\int_T 1_P(gx)dx\\
 &=& \sum_{g\in\mc{G}}\vol(g(T)\cap P)\\
 &=& \vol(P),
\end{eqnarray*}
where again, interchanging summation and integration is justified by the fact that the sum is finite.
 \end{proof}

 \bigskip
\bigskip
 \section{Gauss sums}\label{Gausssums}

 The Weyl group satisfies the following properties:
 \begin{itemize}
  \item it preserves both the Lebesgue and discrete volumes; in particular, it consists of invertible linear
  transformations that preserve the lattice $\ZZ^d$.
  \item it preserves norms, so it also preserves Gauss sums.
 \end{itemize}
 
 It easily follows that the full group $\mc{G}$ also preserves Lebesgue and discrete measures, as well as Gauss
 sums.

 \begin{lemma}\label{xtonx}
  With notation as above, we have
  \[G_P(n)=\sum_{x\in T\cap\frac{1}{n}\ZZ^d}g_P(x)e(n\norm{x}^2).\]
 \end{lemma}
 
 \begin{proof}
  Replacing $x$ by $nx$ in the definition of a Gauss sum, we get
  \begin{eqnarray*}
   G_P(n) &=& \sum_{x\in\frac{1}{n}\ZZ^d}\omega_P(x)e(n\norm{x}^2)\\
   &=& \sum_{x\in T\cap\frac{1}{n}\ZZ^d}\sum_{y\in \mc{G}(x)}\omega_P(y)e(n\norm{x}^2)\\
   &=& \sum_{x\in T\cap\frac{1}{n}\ZZ^d}g_P(x)e(n\norm{x}^2),
  \end{eqnarray*}
since $n\norm{gx}^2\equiv n\norm{x}^2\bmod 1$; indeed, if $gx=wx+\la$, where $w\in\mc{W}$, $\la\in\ZZ^d$, then
\begin{eqnarray*}
 n\norm{wx+\la}^2=n\norm{wx}^2+n\norm{\la}^2+2n\scal{wx,\la}\equiv n\norm{x}^2+2\scal{w(nx),\la}
 \equiv n\norm{x}^2 \bmod1,
\end{eqnarray*}
for all $x\in\frac{1}{n}\ZZ^d$.
 \end{proof}

 \begin{proof}(of Theorem \ref{main1})
  By Proposition \ref{const}, the function $f_P$ is constant and equal to $\abs{\mc{W}}\vol(P)$. So,
  \begin{eqnarray*}
   \vol(P)G(n)^d &=& \vol(P)\sum_{x\in\ZZ^d/\frac{1}{n}\ZZ^d}e(n\norm{x}^2)\\
   &=& \vol(P)\sum_{x\in T\cap\frac{1}{n}\ZZ^d}\frac{1}{\abs{\mc{W}_x}}\sum_{g\in\mc{W}}e(n\norm{x}^2)\\
   &=& \sum_{x\in T\cap\frac{1}{n}\ZZ^d}\frac{\abs{\mc{W}}\vol(P)}{\abs{\mc{G}_x}}e(n\norm{x}^2)\\
   &=& \sum_{x\in T\cap\frac{1}{n}\ZZ^d}g_P(x)e(n\norm{x}^2)\\
   &=& G_P(n),
  \end{eqnarray*}
by Lemma \ref{xtonx} and the fact that
\[g_P(x)=\frac{f_P(x)}{\abs{\mc{G}_x}}=\frac{\abs{\mc{W}}\vol(P)}{\abs{\mc{G}_x}}.\qedhere\]
 \end{proof}
 
 \begin{qu}
  Is the converse true? That is, if $G_P(n)=\vol(P)G(n)^d$ for all $n$, then is it true that $P$ multi-tiles
  the space by $\mc{G}$? 
  %Rather than working with $g_P$, we can work with
  %\[h_P(x)=g_P(x)-\frac{\abs{\mc{W}}\vol(P)}{\abs{\mc{W}_x}},\]
  %so that
  %\[\int_T h_P(x)dx=0\]
  %and the condition $G_P(n)=\vol(P)G(n)^d$ can be rewritten as
  %\begin{equation}\label{converse}
  % \sum_{x\in T\cap\frac{1}{n}\ZZ^d}h_P(x)e(n\norm{x}^2)=0
  %\end{equation}  
  %for all $n$. So the question can be rephrased: if $h_P$ satisfies the above two conditions, is it identically
  %zero?
 \end{qu}
 
 The converse is indeed true for dimensions $d=1, 2$. We have nothing to prove when $d=1$, as any convex lattice polytope in $\RR$ has the form $[a,b]$, where $a,b\in\ZZ$,
and hence multi-tiles $\RR$ $b-a$ times.

The case $d=2$ is quite easy, too. As $P$ can be triangulated, it suffices to prove the converse for lattice triangles. But any lattice
triangle multi-tiles the plane under $\mc{G}$; indeed, suppose that $T=\conv\set{0,v_1,v_2}$, where $v_1,v_2\in\ZZ^2$ are 
linearly independent. The union $T\cup(-T+v_1+v_2)$ is a parallelogram, in particular the closure of a fundamental domain of the 
sublattice of $\ZZ^2$ generated by $v_1$ and $v_2$, which shows that $T$ multi-tiles the plane, therefore any lattice polygon satisfies
the Gauss sum formula and there is nothing else to prove.

 \bigskip
\bigskip
\section{Solid and dihedral angles of a tetrahedron}

Before proceeding to the first $3$-dimensional case, it would be useful to revise a couple of things related to the geometry of the tetrahedron, as well
as the basic tools. Consider the tetrahedron $T$ in $\RR^3$ with vertices $v_0$, $v_1$, $v_2$, and $v_3$. The solid angle at vertex $v_i$ is denoted by $\om_i$
and the dihedral angle at the edge connecting $v_i$ and $v_j$ is denoted by $\om_{ij}$. Here, and throughout the paper, we normalize everything by considering the angles corresponding
 to both $S^1$ and $S^2$ to be equal to $1$ (not $2\pi$ and $4\pi$, respectively).
Under this normalization, we have the \emph{Gram relations} \cite{McMullen71,McMullen77}, which are equalities connecting the solid with the dihedral angles of a tetrahedron: %%cite
  \begin{equation}\label{Gram}
 \om_i=\frac{1}{2}\sum_{j\neq i}\om_{ij}-\frac{1}{4},
 \end{equation}
 which yield
 \[1+\sum_{i=0}^3\om_i=\sum_{0\leq i<j\leq3}\om_{ij}.\]
 We also denote by $n_{ij}=\norm{v_i-v_j}^2$ the squared lengths of the edges. Now let $\set{0,1,2,3}=\set{i,j,k,l}$.
 Oosterom and Strackee \cite{OS83} had proved the following formula for the solid angle of a simple cone:
 \begin{equation}\label{OSgen}
  \cot2\pi\om_i=\frac{\sqrt{n_{ij}n_{ik}n_{il}}+\scal{v_k-v_i,v_l-v_i}\sqrt{n_{ij}}+\scal{v_l-v_i,v_j-v_i}\sqrt{n_{ik}}
  +\scal{v_j-v_i,v_k-v_i}\sqrt{n_{il}}}{\abs{\det(v_j-v_i,v_k-v_i,v_l-v_i)}}.
 \end{equation}
%Similar formulae hold for the other solid angles as well. 
Next, we will focus on the \emph{external solid angles} of a tetrahedron. Unlike the $2$-dimensional case,
there isn't a unique external angle, but three; every external solid angle is detrmined by a vertex and an adjacent edge. The figure below shows us the external
solid angle at $v_0$ with respect to the edge $v_1-v_0$ (for convenience we put $v_0=(0,0,0)$):
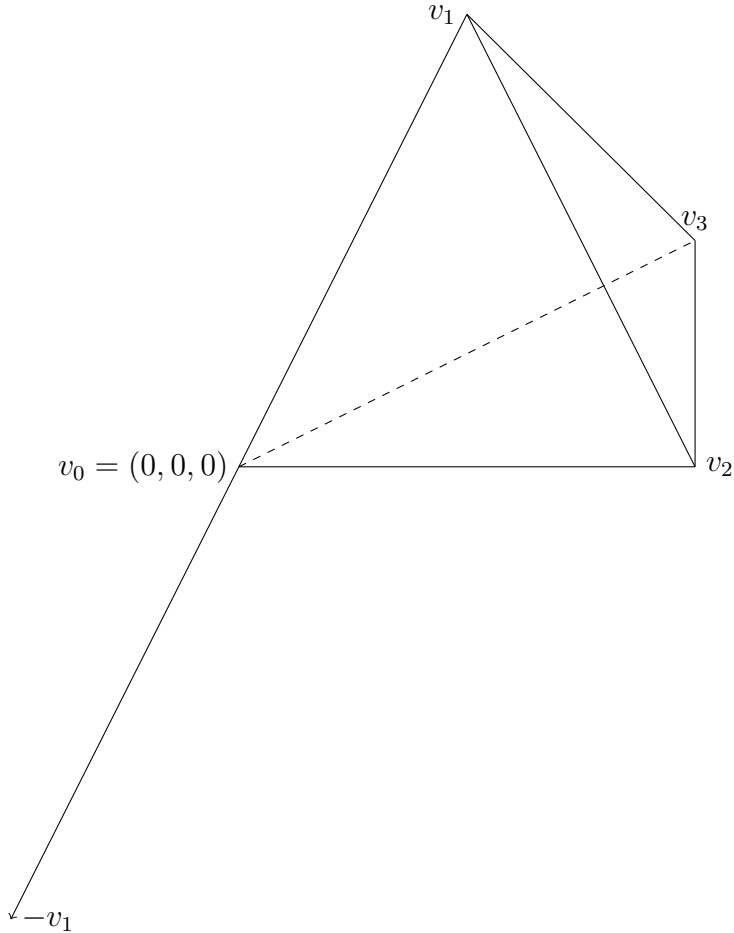
\begin{center}
\begin{tikzpicture}
\draw (0,0) -- (3,6) node[left] {$v_1$};
\draw (0,0) -- (6,0) node[right] {$v_2$};
\draw (3,6) -- (6,0);
\draw (3,6) -- (6,3) node[above] {$v_3$};
\draw (6,0) -- (6,3);
\draw[dashed] (0,0) node[left] {$v_0=(0,0,0)$} -- (6,3);
\draw[->] (0,0) -- (-3,-6) node[right] {$-v_1$};
\end{tikzpicture}
\end{center}
We denote the external solid angle at $v_i$ along $v_j-v_i$ by $\vphi_{ij}$. A basic relation is
\begin{equation}\label{external}
 \om_{ij}=\om_i+\vphi_{ij}.
\end{equation}
The solid angle $\vphi_{ij}$ is defined by the vectors $v_i-v_j,v_k-v_i,v_l-v_i$, and hence
 \begin{equation}\label{OSextgen}
  \cot2\pi\vphi_{ij}=\frac{\sqrt{n_{ij}n_{ik}n_{il}}+\scal{v_k-v_i,v_l-v_i}\sqrt{n_{ij}}-\scal{v_l-v_i,v_j-v_i}\sqrt{n_{ik}}
  -\scal{v_j-v_i,v_k-v_i}\sqrt{n_{il}}}{\abs{\det(v_j-v_i,v_k-v_i,v_l-v_i)}}.
 \end{equation}
 Next, we will make the following assumptions:
 \begin{itemize}
         \item[(a)] $v_0=(0,0,0)$.
         \item[(b)] $v_i\in\ZZ^3$, for all $i$.
         \item[(c)] $T$ has minimal volume, i.~e. $\vol(T)=1/6$, or equivalently, $v_1$, $v_2$, $v_3$ is a basis of $\ZZ^3$.
        \end{itemize}
 Then \eqref{OSgen} and \eqref{OSextgen} become
  \begin{equation}\label{OS}
  \cot2\pi\om_i=\sqrt{n_{ij}n_{ik}n_{il}}+\scal{v_k-v_i,v_l-v_i}\sqrt{n_{ij}}+\scal{v_l-v_i,v_j-v_i}\sqrt{n_{ik}}
  +\scal{v_j-v_i,v_k-v_i}\sqrt{n_{il}},
 \end{equation}
 and
  \begin{equation}\label{OSext}
  \cot2\pi\vphi_{ij}=\sqrt{n_{ij}n_{ik}n_{il}}+\scal{v_k-v_i,v_l-v_i}\sqrt{n_{ij}}-\scal{v_l-v_i,v_j-v_i}\sqrt{n_{ik}}
  -\scal{v_j-v_i,v_k-v_i}\sqrt{n_{il}},
 \end{equation}
respectively. Apparently, $\cot2\pi\om_i$ and $\cot2\pi\vphi_{ij}$ are both algebraic integers, belonging both to the multiquadratic field
$\QQ(\sqrt{n_{ij}},\sqrt{n_{ik}},\sqrt{n_{il}})$, which we denote by $K_i$. Between these two numbers there is a simple algebraic relation.
%(we call an element of a field $K$ \emph{pure} if it doesn't belong to any proper subfield of $K$).
\begin{prop}\label{tau}
 %Both $\cot2\pi\om_i$ and $\cot2\pi\vphi_{ij}$ are pure elements of $K_i$. Furthermore, 
 Suppose that $\sqrt{n_{ij}}\notin\QQ(\sqrt{n_{ik}},\sqrt{n_{il}})$ and 
 $\tau$ is the unique nontrivial $\QQ(\sqrt{n_{ik}},\sqrt{n_{il}})$-automorphism of $K_i$ (i.~e it fixes $\QQ(\sqrt{n_{ik}},\sqrt{n_{il}})$, 
 but $\tau(\sqrt{n_{ij}})=-\sqrt{n_{ij}}$), then $\cot2\pi\vphi_{ij}=-\tau(\cot2\pi\om_i)$ and $\sqrt{n_{ij}}\cot2\pi\om_{ij}\in\QQ(\sqrt{n_{ik}},\sqrt{n_{il}})$.
\end{prop}
\begin{proof}
 The first conclusion is an immediate consequence of \eqref{OS} and \eqref{OSext}. The second follows from \eqref{external} and the formula for the cotangent
 of a sum:
 \[\cot2\pi\om_{ij}=\frac{\cot2\pi\om_i\cot2\pi\vphi_{ij}-1}{\cot2\pi\om_i+\cot2\pi\vphi_{ij}}=
 \frac{-N(\cot2\pi\om_i)-1}{2\sqrt{n_{ij}}(\sqrt{n_{ik}n_{il}}+\scal{v_k-v_i,v_l-v_i})},\]
 hence
 \[\sqrt{n_{ij}}\cot2\pi\om_{ij}=
 \frac{-N(\cot2\pi\om_i)-1}{2(\sqrt{n_{ik}n_{il}}+\scal{v_k-v_i,v_l-v_i})}\in\QQ(\sqrt{n_{ik}},\sqrt{n_{il}}),\]
 where $N$ is the number theoretic norm of the quadratic extension $K_i/\QQ(\sqrt{n_{ik}},\sqrt{n_{il}})$.
\end{proof}

\bigskip
\bigskip
\section{A converse for $3$-dimensional tetrahedra of volume $1/6$}

Assume that
\[G_T(n)=\vol(T)G(n)^d\]
holds for all $n$, for a convex lattice polytope, $T$. 
Any convex polytope is a union of simplices, so it is natural to check whether the converse holds for simplices first. This is the
first nontrivial case as there are lattice tetrahedra that do not satisfy the Gauss sum formula, such as $\conv\set{0,\bs e_1,
\bs e_2,\bs e_3}$, where $\bs e_i$ are the vectors of the standard basis of $\RR^3$.

So, we assume that $T=\conv\set{v_0=0,v_1,v_2,v_3}$ with the additional condition that $T$ has minimal volume.
 This means that $\vol(T)=1/6$ and $v_1$, $v_2$, $v_3$ is a basis of $\ZZ^3$. Let $\om_i$ be the solid angle of $T$ at the vertex
 $v_i$ and $\om_{ij}$ be the dihedral angle at the edge $v_j-v_i$.

 Now let's consider the Gauss sum relations, which for $T$ take the form
 \begin{equation}\label{Gauss}
 \sum_{x\in\ZZ^3}\omega_{nT}(x)e\left(\frac{\norm{x}^2}{n}\right)=\frac{G(n)^3}{6}.
 \end{equation}
 The only lattice points in $T$ are $v_i$ for $0\leq i\leq3$ and their contribution to the Gauss sum is precisely
 $\om_i$ for each $i$, so for $n=1$, \eqref{Gauss} becomes
 \begin{equation}\label{Gauss1}
 \sum_{i=0}^3\om_i=\frac{1}{6}.
 \end{equation}
 In general, the lattice points of $nT$ that lie on the vertices or the edges have the form $av_i+bv_j$ for all $i\neq j$
  where $a+b=n$ with $a,b\geq0$ integers. So, the contribution of these points to $G_T(n)$ is
  \begin{eqnarray*}
  &&\sum_{0\leq i\leq 3}\omega_{nT}(nv_i)e\left(\frac{\norm{nv_i}^2}{n}\right)+
  \sum_{0\leq i<j\leq3}\sum_{\substack{a+b=n\\a,b>0}}\omega_{nT}(av_i+bv_j)e\left(\frac{\norm{av_i+bv_j}^2}{n}\right)\\
  &&=\sum_{0\leq i\leq3}\om_i+
  \sum_{0\leq i<j\leq3}\sum_{a=1}^{n-1}\om_{ij}e\left(\frac{\norm{nv_j+a(v_i-v_j)}^2}{n}\right)\\
  &&=\sum_{0\leq i\leq3}\om_i+
  \sum_{0\leq i<j\leq3}\sum_{a=1}^{n-1}\om_{ij}e\left(\frac{n^2\norm{v_j}^2+2n\scal{v_j,a(v_i-v_j)}+a^2\norm{v_i-v_j}^2}{n}\right)\\
  &&=\sum_{0\leq i\leq3}\om_i+
  \sum_{0\leq i<j\leq3}\sum_{a=1}^{n-1}\om_{ij}e\left(\frac{a^2\norm{v_i-v_j}^2}{n}\right)\\
  &&=\sum_{0\leq i\leq3}\om_i+
  \sum_{0\leq i<j\leq3}\om_{ij}[G(n_{ij},n)-1]\\ %with \
  &&=-1+\sum_{0\leq i<j\leq3}\om_{ij}G(n_{ij},n),
  \end{eqnarray*} 
  using \eqref{Gram},
 where we put $n_{ij}=\norm{v_j-v_i}^2$, the squared lengths of the edges, and $G(a,b)$ is the quadratic Gauss sum given by
 \[G(a,b)=\sum_{n=0}^{b-1}e\left(\frac{an^2}{b}\right).\]
 The following formula by Gauss \cite{NZM} for $\gcd(a,b)=1$ will be very useful:
 \begin{equation}\label{quadGauss}
 G(a,b)=\begin{cases}
 0, \	&b\equiv2\bmod4\\
 \eps_b\sqrt{b}\left(\frac{a}{b}\right), \	&b \text{ odd}\\
 (1+i)\eps_a^{-1}\sqrt{b}\left(\frac{b}{a}\right), \	&4|b
 \end{cases}
 \end{equation}
 where
 \[\eps_m=\begin{cases}
 1, \	&m\equiv1\bmod4\\
 i, \	&m\equiv3\bmod4
 \end{cases}\]
 and $\left(\frac{a}{b}\right)$ is the Jacobi symbol. For $\gcd(a,b)=d>1$ we simply have $G(a,b)=dG(a/d,b/d)$.
 If $x$ is any other lattice point in $nT$, then we have 
 $\omega_{nT}(x)=1/2$ when $x$ is in the relative interior of one facet, and $\omega_{nT}(x)=1$ when $x\in\inter(T)$. This yields:
 
 \begin{prop}\label{dihedral}
 Let $T$ be a lattice tetrahedron with vertices $v_i$, $0\leq i\leq3$. Let $\om_{ij}$ be the dihedral angle at the edge
 $v_i-v_j$ and let $n_{ij}=\norm{v_i-v_j}^2$. Then
 \[G_T(n)=-1+\sum_{0\leq i<j\leq3}\om_{ij}G(n_{ij},n)+\ka(n),\]
 where $\ka(n)\in\QQ(e(1/n))$. 
 \end{prop}
 
 \begin{rem}
 The above holds for all lattice tetrahedra, not just the ones with minimal volume. However, if $\vol(T)=1/6$, then the only lattice
 points of $2T$ are the vertices and the midpoints of the edges, therefore $\ka(1)=\ka(2)=0$. The explicit formula for
 $\ka(n)$ is
 \[\ka(n)=\frac{1}{2}\sum_{0\leq i<j<k\leq3}\sum_{\substack{a+b+c=n\\a,b,c>0}}e\left(\frac{\norm{av_i+bv_j+cv_k}^2}{n}\right)
 +\sum_{\substack{a+b+c+d=n\\a,b,c,d>0}}e\left(\frac{\norm{av_0+bv_1+cv_2+dv_3}^2}{n}\right).\]
 In particular,
 \[\ka(3)=\frac{1}{2}\sum_{0\leq i<j<k\leq3}e\left(\frac{\norm{v_i+v_j+v_k}^2}{3}\right)\]
 and
 \begin{eqnarray*}
 \ka(4) &=& \frac{1}{2}\sum_{\substack{0\leq i<j\leq 3\\k\neq i,j}}e\left(\frac{\norm{v_i+v_j+2v_k}^2}{4}\right)
 +e\left(\frac{\norm{v_0+v_1+v_2+v_3}^2}{4}\right)\\
 &=& \sum_{0\leq i<j\leq3}e\left(\frac{\norm{v_i+v_j}^2}{4}\right)
 +e\left(\frac{\norm{v_0+v_1+v_2+v_3}^2}{4}\right)
 \end{eqnarray*}
 \end{rem}
 
 Next, we will investigate the parity of $n_{ij}$. Since $\vol(T)=1/6$, any three vectors corresponding to edges at a common vertex
 of $T$ form a basis of $\ZZ^3$. Furthermore, if $x=(x_1,x_2,x_3)\in\ZZ^3$ then $\norm{x}^2\equiv x_1+x_2+x_3\bmod2$, so if
 $\norm{u}^2$ and $\norm{v}^2$ have the same parity, then $\norm{u-v}^2$ is even. This means that at any face of $T$, either all
 or exactly one edge has even squared length. Moreover, not all three squared lengths of edges with a common vertex can be even,
 otherwise these vectors would span a proper even sublattice of $\ZZ^3$. Thus, we have one of the following two situations for the
 edges with even squared lengths of $T$: either they form a triangle, or they are opposite, having no vertex in common. By an appropriate
 lattice translation of $T$, we may assume that $v_0=0$, $n_{01}=\norm{v_1}^2$ and $n_{03}=\norm{v_3}^2$ are odd.
 
 \bigskip
 \noindent
 $\boxed{n_{02}=\norm{v_2}^2\text{ is odd}}$ Then $n_{ij}$ for $1\leq i<j\leq3$ are even. Then by Proposition \ref{dihedral} 
 and \eqref{quadGauss} we get
 \[G_T(2)=-1+2(\om_{12}+\om_{13}+\om_{23}).\]
 From \eqref{Gram} and \eqref{Gauss1} we get 
 \begin{equation}\label{sumdihedral}
 \sum_{0\leq i<j\leq3}\om_{ij}=\frac{7}{6},
 \end{equation}
 and since $G_T(2)=0$, as $T$ satisfies the Gauss sum formula for all $n$, we get
 \[\om_{12}+\om_{13}+\om_{23}=\frac{1}{2},\]
 and
 \begin{equation}\label{oddsum}
 \om_{01}+\om_{02}+\om_{03}=\frac{2}{3},
 \end{equation}
 hence
 \begin{equation}\label{solid0}
 \om_0=\frac{1}{12},
 \end{equation}
 by virtue of \eqref{Gram}.
 
 Next, we wish to examine the possible values of $n_{ij}\bmod4$. For the even $n_{ij}$, it is not hard to see that
 $n_{ij}\equiv2\bmod4$, because the edges $v_i-v_j$ correspond to primitive vectors in $\ZZ^3$; if $4|x_1^2+x_2^2+x_3^2$ then
 all $x_i$ must be even. The residue $n_{0i}\bmod4$ depends on the parity of the coordinates of $v_i$. First we notice that no
 two of the $n_{0i}$ can be $3\bmod4$; if, for example, $n_{01}\equiv n_{02}\equiv3\mod4$, then all coordinates of $v_1$ and
 $v_2$ must be odd, which yields $\frac{1}{2}(v_1+v_2)\in\ZZ^3$, a contradiction, because $v_1$, $v_2$, $v_3$ is a basis
 of $\ZZ^3$. We have thus proven:
 
 \begin{prop}
 Let $v_1$, $v_2$, $v_3$ be a basis of $\ZZ^3$ such that all $\norm{v_i}^2$ are odd. Then at most one of the $\norm{v_i}^2$ is
 $3\bmod4$.
 \end{prop}
 
 We will show that if $n_{0i}\equiv1\bmod4$ for all $i$, then $T$ cannot satisfy the Gauss sum relation for $n=4$. In this
 case, each $v_i$ has exactly one odd coordinate and two even. Since $\frac{1}{2}(v_i+v_j)\notin\ZZ^3$, different coordinates
 in the vectors $v_i$ are odd (or in simple terms, the entries $\bmod2$ of the matrix whose columns are $v_i$ is equal to the
 identity matrix). This shows that the coordinates of $v_1+v_2+v_3$ are all odd. Therefore,
 \[\ka(4)=\sum_{0\leq i<j\leq3}e\left(\frac{n_{ij}}{4}\right)+e\left(\frac{3}{4}\right)=-3+2i.\]
 Since $n_{ij}\equiv2\bmod4$ for $1\leq i<j\leq3$, we have $G(n_{ij},4)=2G(n_{ij}/2,2)=0$ by \eqref{quadGauss} and \ref{dihedral}
 we get
 \[G_T(4)=-1+\sum_{i=1}^3\om_{0i}G(4)+\ka(4)=-1+\frac{2}{3}\cdot2(1+i)-3+2i=-\frac{8}{3}+\frac{10}{3}i,\]
 while by 
 \eqref{quadGauss} again we have
 \begin{equation}
 \vol(T)G(4)^3=\frac{1}{6}[2(1+i)]^3=\frac{8}{3}(-1+i)\neq G_T(4).
 \end{equation}
 
 Hence, we may assume that $n_{03}\equiv3\bmod4$, while $n_{01}\equiv n_{02}\equiv1\bmod4$. It is not hard to see that 
 $\norm{v_1+v_2+v_3}^2\equiv1\bmod4$. Therefore,
 \[\ka(4)=\sum_{0\leq i<j\leq3}e\left(\frac{n_{ij}}{4}\right)+e\left(\frac{1}{4}\right)=-3+2i\]
 and
 \begin{eqnarray*}
 G_T(4) &=& -1+(\om_{01}+\om_{02})G(4)+\om_{03}G(3,4)-3+2i\\
 &=& [2(\om_{01}+\om_{02}+\om_{03})-4]+[2(\om_{01}+\om_{02}-\om_{03})+2]i\\
 &=& -\frac{8}{3}+\left[\frac{10}{3}-4\om_{03}\right]i,
 \end{eqnarray*}
 by \eqref{oddsum}, therefore $\om_{03}=1/6$ since $\vol(T)G(4)^3=\frac{8}{3}(-1+i)$. We also get $\om_{01}+\om_{02}=1/2$ from
 \eqref{oddsum}.
 
 Applying \eqref{OS} for $i=0$ we get
 \begin{equation}\label{OSom0}
 \cot2\pi\om_0=\sqrt{n_{01}n_{02}n_{03}}+\scal{v_1,v_2}\sqrt{n_{03}}
 +\scal{v_2,v_3}\sqrt{n_{01}}+\scal{v_3,v_1}\sqrt{n_{02}},
 \end{equation}
 so by \eqref{solid0} we get
 \begin{equation}\label{tansolid}
 \sqrt{3}=\sqrt{n_{01}n_{02}n_{03}}+\scal{v_1,v_2}\sqrt{n_{03}}
 +\scal{v_2,v_3}\sqrt{n_{01}}+\scal{v_3,v_1}\sqrt{n_{02}}.
 \end{equation}
 Let $K=\QQ(\sqrt{n_{01}},\sqrt{n_{02}})$. Since $n_{01}\equiv n_{02}\equiv1\bmod4$, we have $\sqrt{q}\notin K$ for any $q\equiv3\bmod4$.
 This is trivial if $K=\QQ$, as $q$ cannot be a square. If $\sqrt{q}\in K\neq\QQ$, then $\QQ(\sqrt{q})$ is a quadratic
 subfield of $K$. The quadratic subfields are exactly $\QQ(\sqrt{n_{01}})$, $\QQ(\sqrt{n_{02}})$, and $\QQ(\sqrt{n_{01}n_{02}})$
 (they coincide if $[K:\QQ]=2$), which yields that $q$ has the same square-free part with one of $n_{01}$, $n_{02}$, $n_{01}n_{02}$,
 but this is impossible as $q\equiv3\bmod4$ while $n_{01}\equiv n_{02}\equiv n_{01}n_{02}\equiv1\bmod4$. Therefore,
 $[K(\sqrt{n_{03}}):K]=2$, and $1, \sqrt{n_{03}}$ is a $K$-basis of $K(\sqrt{n_{03}})$. As $\sqrt{3}\in K(\sqrt{n_{03}})\sm K$ by
 \eqref{tansolid}, we get $\sqrt{3}=a+b\sqrt{n_{03}}$ for some $a,b\in K$ with $b\neq0$. Squaring both sides we obtain
 $3=a^2+b^2n_{03}+2ab\sqrt{n_{03}}$, so we must have $a=0$.
 Again, by \eqref{tansolid} we get
 \[\scal{v_2,v_3}\sqrt{n_{01}}+\scal{v_3,v_1}\sqrt{n_{02}}=0.\]
 If $n_{01}$ and $n_{02}$ do not have the same square-free part, then $\sqrt{n_{01}}$ and $\sqrt{n_{02}}$ are linearly independent
 over $\QQ$, so we must have
 \[\scal{v_2,v_3}=\scal{v_3,v_1}=0,\]
 a contradiction, since
 \[2\scal{v_2,v_3}=n_{02}+n_{03}-n_{23}\equiv2\bmod4.\]
 So $n_{01}$ and $n_{02}$ have the same square-free part, hence $\sqrt{n_{01}n_{02}}\in\ZZ$, and by \eqref{tansolid} we obtain
 \[\sqrt{3}=(\sqrt{n_{01}n_{02}}+\scal{v_1,v_2})\sqrt{n_{03}}.\]
 Since $\sqrt{n_{03}}\geq3$ and $\sqrt{n_{01}n_{02}}+\scal{v_1,v_2}\geq1$ (as an integer), we must have equality in both cases, which
 yields $n_{03}=3$. 
 \begin{prop}\label{vectors}
  With notation as above, let $n_{03}=3$, and assume that $\om_{03}=1/6$. Then, up to an appropriate action of $\mc{W}$, we may assume that
  \[v_1=(k+1,k,k),\	v_2=(l,l,l-1),\	v_3=(1,1,1).\]
 \end{prop}
 
 \begin{proof}
 Applying an appropriate reflection from the group $\mc{W}$, we may assume without loss of generality that
 \[v_3=(1,1,1).\]
 Now consider the hyperplane $H=v_3^{\perp}$, and let $\La$ be the orthogonal projection of $\ZZ^3$ onto $H$. It is not hard to see
 that $\La$ is isomorphic to the hexagonal lattice, and the vectors of smallest length are $\pi(\pm\bs e_i)$, where
 $\pi:\RR^3\rightarrow H$ is the orthogonal projection. By hypothesis, $\pi(v_1)$ and $\pi(v_2)$ is a basis of $\La$ and the angle
 between these two vectors is $\pi/3$ by $\om_{03}=1/6$, therefore they must be of smallest length. Permutations of coordinates
 of $\RR^3$ correspond to rotations of $H$ by multiples of $\pi/3$ or reflections along $\pi(\bs e_i)$, so without loss of
 generality we may assume that $\pi(v_1)=\pi(\bs e_1)$ and $\pi(v_2)=\pi(-\bs e_3)$, hence
 \[v_1=(k+1,k,k),\	v_2=(l,l,l-1).\qedhere\]
 \end{proof}
 From Proposition \ref{vectors} and the fact that $n_{01}$ and $n_{02}$ are odd, follows that $k$ and $l$ are even in our case.
 Since $\om_{01}+\om_{02}=1/2$, we will have
 \[\cos2\pi\om_{01}+\cos2\pi\om_{02}=0.\]
 But
 \[\cos2\pi\om_{01}=\frac{\scal{v_1\times v_2,v_1\times v_3}}{\norm{v_1\times v_2}\cdot\norm{v_1\times v_3}}
 =\frac{\scal{(-k,k-l+1,l),(0,-1,1)}}{\sqrt{2}\sqrt{k^2+l^2+(k-l+1)^2}}=\frac{-k+2l-1}{\sqrt{2k^2+2l^2+2(k-l+1)^2}},\]
 while
 \[\cos2\pi\om_{02}=\frac{\scal{v_2\times v_1,v_2\times v_3}}{\norm{v_2\times v_1}\cdot\norm{v_2\times v_3}}
 =\frac{\scal{(k,-k+l-1,-l),(1,-1,0)}}{\sqrt{2}\sqrt{k^2+l^2+(k-l+1)^2}}=\frac{2k-l+1}{\sqrt{2k^2+2l^2+2(k-l+1)^2}},\]
 therefore we must have $k=l$, hence
 \[v_2=(k,k,k-1).\]
 As we've seen above, $n_{01}=3k^2+2k+1$ and $n_{02}=3k^2-2k+1$ must have the same square free part, say $d$. But this $d$ is odd and
 also a common divisor of $n_{01}$ and $n_{02}$, therefore $d|n_{01}-n_{02}=4k$, so $d|k$. Since $\gcd(k,n_{01})=1$, we must have
 $d=1$, so $n_{01}$ and $n_{02}$ are both perfect (odd) squares. Let $m,n\geq0$ be such that
 \begin{equation}\label{kmn}
 \begin{split}
 3k^2+2k+1 &= (2m+1)^2\\ 
 3k^2-2k+1 &= (2n+1)^2,
 \end{split}
 \end{equation}
 which yields
 \[k=(m-n)(m+n+1).\]
 Adding the equations \eqref{kmn} we get 
 \[3k^2=2(m^2+n^2+m+n).\]
 If $m\neq n$, we obtain
 \[3k^2\geq3(m+n+1)^2=3(m^2+n^2+1+2mn+2m+2n)>2(m^2+n^2+m+n),\]
 so we must have $m=n$ and $k=0$. Therefore,
 \[v_1=(1,0,0),\	v_2=(0,0,-1).\]
 Next, we will verify that the Gauss sum relation for $n=3$ fails. We have
 \[n_{01}=n_{02}=1, n_{03}=2, n_{12}=n_{13}=2, n_{23}=6,\]
 and
 \begin{equation*}%\label{vol3}
 \vol(T)(G(3))^3=-i\frac{\sqrt{3}}{2},
 \end{equation*}
 but
 \begin{eqnarray*}
 G_T(3) &=& -1+\sum_{0\leq i<j\leq3}\om_{ij}G(n_{ij},3)+\ka(3)\\
 &=& -1+(\om_{01}+\om_{02})G(3)+(\om_{12}+\om_{13})G(2,3)+3(\om_{03}+\om_{23})+\frac{1}{2}[1+3e(2/3)]\\
 &=& -1+i\frac{\sqrt{3}}{2}-(1/2-\om_{23})i\sqrt{3}+3(1/6+\om_{23})+\frac{1}{2}\left[-2-i\frac{3\sqrt{3}}{2}\right]\\
 &=& 3(\om_{23}-1/2)+(\om_{23}-5/4)i\sqrt{3}.
 \end{eqnarray*}
 Taking real and imaginary parts, if $G_T(3)=\vol(T)(G(3))^3$ then we should have simultaneously have $\om_{23}=1/2$ and
 $\om_{23}=3/4$, an absurdity. We thus conclude that:
 
 \begin{prop}
 Let $T=\conv(0,v_1,v_2,v_3)$ with $v_1$, $v_2$, $v_3$ basis of $\ZZ^3$, such that all $\norm{v_i}^2$ are odd. Then $T$ cannot satisfy
 the Gauss sum relations. In particular, $G_T(n)=\vol(T)(G(n))^3$ fails for some $n\leq4$.
 \end{prop}
 
 \bigskip
 \noindent
 $\boxed{n_{02}=\norm{v_2}^2\text{ is even}}$ Then, $n_{02}\equiv n_{13}\equiv2\bmod4$, and all other $n_{ij}$ are odd. As we have
 already seen, two adjacent edges cannot have both squared length $3\bmod4$, so there are at most two of them in $T$. So, we may
 assume that $v_1\equiv(1,0,0)\bmod2\ZZ^3$. If $n_{03}\equiv1\bmod4$, then, up to the action of group $\mc{W}$ and possibly
 interchanging $v_1$ and $v_3$ we will have
 \[A\equiv\begin{pmatrix}
 1 & 1 & 0\\
 0 & 1 & 0\\
 0 & 0 & 1
 \end{pmatrix},\]
 if we consider the entries of $A=(v_1^T\  v_2^T\  v_3^T)$ taken $\bmod2$. Then, it is clear that exactly one edge satisfies
 $n_{ij}\equiv3\bmod4$, in particular $n_{23}$. If $n_{03}\equiv3\bmod4$, then again, up to the action of $\mc{W}$ we will have
 \begin{equation}\label{matrix}
 A\equiv\begin{pmatrix}
 1 & 1 & 1\\
 0 & 1 & 1\\
 0 & 0 & 1
 \end{pmatrix},
 \end{equation}
 and again, only one edge satisfies $n_{ij}\equiv3\bmod4$, this time $n_{03}$. So, in any case, there is exactly one edge satisfying
 $n_{ij}\equiv3\bmod4$, and after an appropriate lattice translation, we can always take $n_{03}$ to be that edge. Without loss
 of generality, $A$ satisfies \eqref{matrix} and we have
 \begin{equation}\label{lengths}
 n_{01}\equiv n_{12}\equiv n_{23}\equiv1\bmod4, \	n_{02}\equiv n_{13}\equiv2\bmod4, \	n_{03}\equiv3\bmod4,
 \end{equation}
 or more succinctly,
 \[n_{ij}\equiv j-i\bmod4.\]
 Also from \eqref{matrix} we get that
 \begin{equation}\label{innerprods}
 \scal{v_1,v_2}\text{ and }\scal{v_1,v_3}\text{ are odd, while }\scal{v_2,v_3}\text{ is even}. 
 \end{equation}
 By Proposition \ref{dihedral} and \eqref{lengths}, the Gauss sum relation for $n=2$ becomes
 \[0=G_T(2)=-1+2(\om_{02}+\om_{13}),\]
 therefore,
 \begin{equation}\label{sumeven}
 \om_{02}+\om_{13}=\frac{1}{2},
 \end{equation}
 and
 \begin{equation}\label{sumodd}
 \om_{01}+\om_{12}+\om_{23}+\om_{03}=\frac{2}{3},
 \end{equation}
 because of \eqref{sumdihedral}.
 By \eqref{matrix} and \eqref{lengths} we get
 \[\ka(4)=\sum_{0\leq i<j\leq3}e\left(\frac{n_{ij}}{4}\right)+e\left(\frac{\norm{v_1+v_2+v_3}^2}{4}\right)=-3+2i,\]
 hence Proposition \ref{dihedral} for $n=4$ yields
 \begin{eqnarray*}
 G_T(4) &=& -1+(\om_{01}+\om_{12}+\om_{23})G(4)+\om_{03}G(3,4)-3+2i\\
 &=& -4+2(\om_{01}+\om_{12}+\om_{23}+\om_{03})+2(\om_{01}+\om_{12}+\om_{23}-\om_{03})i+2i\\
 &=& -\frac{8}{3}+\left(\frac{10}{3}-4\om_{03}\right)i,
 \end{eqnarray*}
 while $\vol(T)(G(4))^3=\frac{8}{3}(-1+i)$, so if the Gauss sum relation holds for $n=4$, then we get
 \begin{equation}\label{sum3mod4}
 \om_{03}=\frac{1}{6},
 \end{equation}
 and
 \begin{equation}\label{sum1mod4}
 \om_{01}+\om_{12}+\om_{23}=\frac{1}{2},
 \end{equation}
 by \eqref{sumodd}. Next, we consider again the orthogonal projection $\pi:\RR^3\rightarrow H$, where $H=v_3^{\perp}$ and put
 $\La=\pi(\ZZ^3)$. The vectors $\pi(v_1)$ and $\pi(v_2)$ is a basis of $\La$ and the angle between them is equal to the dihedral angle
 $\om_{03}$. However, the lattice $\La$ contains also vectors orthogonal to $\pi(v_1)$, namely $v_3\times v_1$, so let
 $a\pi(v_1)+b\pi(v_2)$ be orthogonal to $\pi(v_1)$, with $a,b\in\ZZ$ nonzero. Hence, the orthogonal projection of $b\pi(v_2)$ on
 $\RR\pi(v_1)$ is equal to $-a\pi(v_1)$, therefore $\norm{-a\pi(v_1)}=\frac{1}{2}\norm{b\pi(v_2)}$ or
 \[\norm{\pi(v_2)}=\abs{\frac{a}{2b}}\norm{\pi(v_1)}.\]
 Since $v_1$, $v_2$, $v_3$ is a basis of $\ZZ^3$ we have
 \begin{eqnarray*}
 1 &=& \abs{\scal{v_3,v_1\times v_2}}=\abs{\scal{v_3,\pi(v_1)\times\pi(v_2)}}=\sqrt{n_{03}}\norm{\pi(v_1)\times\pi(v_2)}\\
 &=& \sqrt{n_{03}}\norm{\pi(v_1)}\norm{\pi(v_2)}\sin2\pi\om_{03}=\sqrt{3n_{03}}\abs{\frac{a}{4b}}\norm{\pi(v_1)}^2,
 \end{eqnarray*}
 and since $\abs{\frac{a}{4b}}\norm{\pi(v_1)}^2\in\QQ$ we must have 
 \begin{equation}\label{n03}
  n_{03}=3m^2,
 \end{equation}
  for some $m\in\ZZ$.
  
  %\begin{tikzpicture}
  
   %\filldraw [gray] (0,0) circle [radius=2pt]
								%(1,1) circle [radius=2pt]
%(2,1) circle [radius=2pt]
%(2,0) circle [radius=2pt];
% \draw (0,0) .. controls (1,1) and (2,1) .. (2,0);

 % \end{tikzpicture}

  Next, the Gram relations \eqref{Gram} along with \eqref{sumeven} and \eqref{sum1mod4} form a system of six linear equations in terms of the dihedral angles
  $\om_{ij}$. This system has a unique solution, namely,
  \begin{eqnarray}
\label{01}   \om_{01} &=& \phantom{-}\tfrac{1}{2}\om_0-\tfrac{1}{2}\om_1-\tfrac{3}{2}\om_2-\tfrac{1}{2}\om_3+\tfrac{1}{4}\\ \label{02}
   \om_{02} &=& \phantom{-}\tfrac{1}{2}\om_0-\tfrac{1}{2}\om_1+\tfrac{1}{2}\om_2-\tfrac{1}{2}\om_3+\tfrac{1}{4}\\ \label{03}
   \om_{03} &=& \phantom{-\tfrac{1}{2}}\om_0+\phantom{\tfrac{1}{2}}\om_1+\phantom{\tfrac{1}{2}}\om_2+\phantom{\tfrac{1}{2}}\om_3\\ \label{12}
   \om_{12} &=& \phantom{-\tfrac{1}{2}\om_1+\ \ }2\om_1+2\om_2\\ \label{13}
   \om_{13} &=& -\tfrac{1}{2}\om_0+\tfrac{1}{2}\om_1-\tfrac{1}{2}\om_2+\tfrac{1}{2}\om_3+\tfrac{1}{4}\\ \label{23}
   \om_{23} &=& -\tfrac{1}{2}\om_0-\tfrac{3}{2}\om_1-\tfrac{1}{2}\om_2+\tfrac{1}{2}\om_3+\tfrac{1}{4}.
  \end{eqnarray}
Formulae \eqref{02} and \eqref{13} along with \eqref{Gauss1} yield
\begin{equation}\label{02-13}
 \om_{02}-\om_0-\om_2=\om_{13}-\om_1-\om_3=1/6.
\end{equation}
In order to visualize $\om_{02}-\om_0-\om_2$, we consider $T$ and its translate $T-v_2$, as in the figure below. 
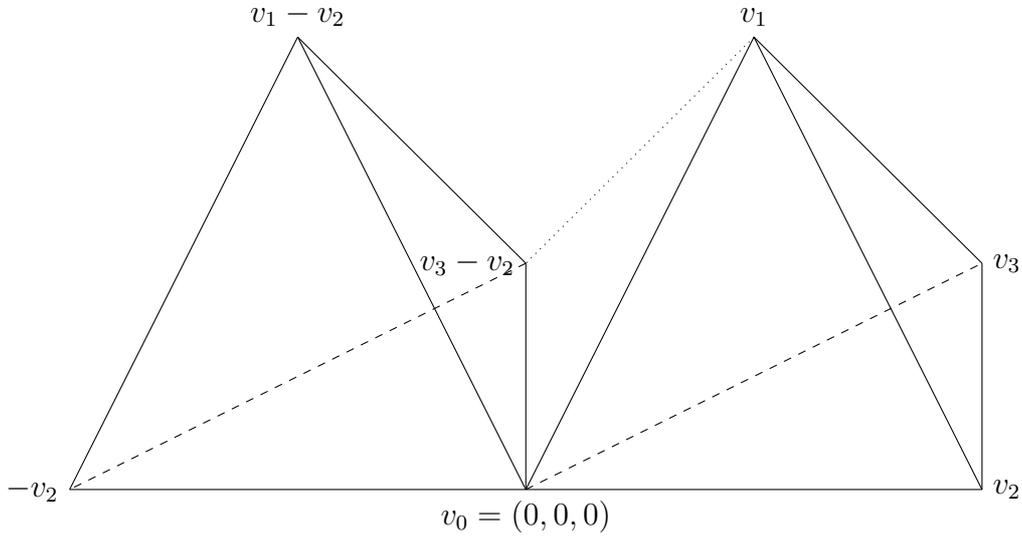
\begin{center}
\begin{tikzpicture}
\draw (0,0) node[below] {$v_0=(0,0,0)$} -- (3,6) node[above] {$v_1$};
\draw (0,0) -- (6,0) node[right] {$v_2$};
\draw (3,6) -- (6,0);
\draw (3,6) -- (6,3) node[right] {$v_3$};
\draw (6,0) -- (6,3);
\draw[dashed] (0,0) -- (6,3);
\draw (0,0) -- (-6,0) node[left] {$-v_2$};
\draw (0,0) -- (-3,6) node[above] {$v_1-v_2$};
\draw (-6,0) -- (-3,6);
\draw (0,0) -- (0,3) node[left] {$v_3-v_2$};
\draw (0,3) -- (-3,6);
\draw[dashed] (0,3) -- (-6,0);
\draw[dotted] (3,6) -- (0,3);
\end{tikzpicture}
\end{center}
As can be seen, $\om_{02}-\om_0-\om_2$ is the solid angle of the cone with vectors $v_1$, $v_3$, $v_3-v_2$, $v_1-v_2$, which we divide into two simplicial
cones, one with vectors $v_3$, $v_3-v_2$, $v_1$, and one with $v_3-v_2$, $v_1-v_2$, $v_1$. We denote the solid angles by $\Om_1$, $\Om_2$,
respectively. Then, from \eqref{02-13} we get
\begin{equation}\label{Om}
 \Om_1+\Om_2=\frac{1}{6}.
\end{equation}
By \eqref{OS} we get
\begin{equation}
\cot2\pi\Om_1=\sqrt{n_{01}n_{23}n_{03}}+\scal{v_3,v_3-v_2}\sqrt{n_{01}}+\scal{v_3-v_2,v_1}\sqrt{n_{03}}+\scal{v_1,v_3}\sqrt{n_{23}} 
\end{equation}
and
\begin{equation}\label{OSOm2}
 \cot2\pi\Om_2=\sqrt{n_{01}n_{12}n_{23}}+\scal{v_3-v_2,v_1-v_2}\sqrt{n_{01}}+\scal{v_1-v_2,v_1}\sqrt{n_{23}}+\scal{v_1,v_3-v_2}\sqrt{n_{12}}.
\end{equation}
Put $K=\QQ(\sqrt{n_{01}},\sqrt{n_{12}},\sqrt{n_{23}})$. By \eqref{lengths} we have $\sqrt{n_{03}}\notin K$. We observe that $\cot2\pi\Om_2\in K$ and
by \eqref{Om} we have
\begin{equation}
 \frac{1}{\sqrt{3}}=\frac{\cot2\pi\Om_1\cot2\pi\Om_2-1}{\cot2\pi\Om_1+\cot2\pi\Om_2}
\end{equation}
or equivalently
\begin{equation}\label{cotOm}
 \cot2\pi\Om_1+\cot2\pi\Om_2=\sqrt{3}\cot2\pi\Om_1\cot2\pi\Om_2-\sqrt{3}.
\end{equation}
As $1$ and $\sqrt{3}$ are $K$-linearly independent we get
\begin{equation}\label{cotOm2}
 \cot2\pi\Om_2=\frac{\scal{v_3,v_3-v_2}\sqrt{n_{01}}+\scal{v_1,v_3}\sqrt{n_{23}}}{3m(\sqrt{n_{01}n_{23}}+\scal{v_3-v_2,v_1})-1}=
 \frac{m(\sqrt{n_{01}n_{23}}+\scal{v_3-v_2,v_1})+1}{\scal{v_3,v_3-v_2}\sqrt{n_{01}}+\scal{v_1,v_3}\sqrt{n_{23}}},
\end{equation}
by \eqref{cotOm} and \eqref{n03}. \eqref{cotOm2} yields $\cot2\pi\Om_2\in\QQ(\sqrt{n_{01}},\sqrt{n_{23}})$, and then by \eqref{OSOm2} we get
\begin{equation}\label{lowdegree1}
 \sqrt{n_{12}}\in\QQ(\sqrt{n_{01}},\sqrt{n_{23}}).
\end{equation}
Indeed, if $\sqrt{n_{12}}\notin\QQ(\sqrt{n_{01}},\sqrt{n_{23}})$, then $1$ and $\sqrt{n_{12}}$ are $\QQ(\sqrt{n_{01}},\sqrt{n_{23}})$-linearly independent, and the coefficient
of $\sqrt{n_{12}}$ in \eqref{OSOm2} is $\sqrt{n_{01}n_{23}}+\scal{v_1,v_3-v_2}$ which is nonzero, since $v_1$ and $v_3-v_2$ are not parallel. This would yield
$\cot2\pi\Om_2\notin\QQ(\sqrt{n_{01}},\sqrt{n_{23}})$, a contradiction.

Combining the two equations in \eqref{cotOm2} we get
\begin{equation}
 \cot^22\pi\Om_2=\frac{m(\sqrt{n_{01}n_{23}}+\scal{v_3-v_2,v_1})+1}{3m(\sqrt{n_{01}n_{23}}+\scal{v_3-v_2,v_1})-1},
\end{equation}
so $\cot^22\pi\Om_2\in\QQ(\sqrt{n_{01}n_{23}})$ and by \eqref{OSOm2} $\cot^22\pi\Om_2$ is an algebraic integer.

\begin{prop}\label{n03=3}
 If $\sqrt{n_{01}n_{23}}\in\QQ$ then $m=1$, hence $n_{03}=3$ and $\cot2\pi\Om_2=1$, hence $\Om_2=1/8$ and $\Om_1=1/24$.
\end{prop}

\begin{proof}
 If $\sqrt{n_{01}n_{23}}\in\QQ$ then $\cot^22\pi\Om_2\in\ZZ$. Put $z=\sqrt{n_{01}n_{23}}+\scal{v_3-v_2,v_1}$. By Cauchy-Schwarz inequality we have
 $z>0$, and since $z\in\ZZ$ we must have $z\geq1$. Then
 \[\frac{mz+1}{3mz-1}\geq1,\]
 whence $mz\leq1$, thus $m=z=1$, which proves that $n_{03}=3$ and $\cot2\pi\Om_2=1$, hence $\Om_2=1/8$. Finally, by \eqref{Om} we get $\Om_1=1/24$. 
\end{proof}

Our goal is to show that the hypothesis of this Proposition is true. The next equation that we'll investigate is
\begin{equation}\label{om02om2om01}
 \om_{02}-2\om_2=\om_{01},
\end{equation}
which follows from \eqref{01} and \eqref{02}. From \eqref{external} we then get
\begin{equation}\label{om02-2om2}
\om_{02}-2\om_2=\vphi_{20}-\om_2,
\end{equation}
and 
\begin{equation}\label{om01}
\om_{01}=\om_0+\vphi_{01},
\end{equation}
hence
\begin{equation}
 \cot2\pi(\vphi_{20}-\om_2)=\cot2\pi(\om_0+\vphi_{01}).
\end{equation}
Applying \eqref{OS} and \eqref{OSext} accordingly we have
\begin{equation}\label{OSphi01}
\cot2\pi\vphi_{01}=\sqrt{n_{01}n_{02}n_{03}}+\scal{v_2,v_3}\sqrt{n_{01}}-\scal{v_3,v_1}\sqrt{n_{02}}-\scal{v_1,v_2}\sqrt{n_{03}}
\end{equation}
\begin{equation}\label{OSphi20}
\cot2\pi\vphi_{20}=\sqrt{n_{12}n_{02}n_{23}}+\scal{v_2,v_3-v_2}\sqrt{n_{12}}+\scal{v_3-v_2,v_1-v_2}\sqrt{n_{02}}+\scal{v_1-v_2,v_2}\sqrt{n_{23}}
\end{equation}
\begin{equation}\label{OSom2}
\cot2\pi\om_2=\sqrt{n_{12}n_{02}n_{23}}-\scal{v_2,v_3-v_2}\sqrt{n_{12}}+\scal{v_3-v_2,v_1-v_2}\sqrt{n_{02}}-\scal{v_1-v_2,v_2}\sqrt{n_{23}}.
\end{equation}
Then by \eqref{om02-2om2}, \eqref{om01}, \eqref{OSom0}, \eqref{OSphi01}, \eqref{OSphi20}, \eqref{OSOm2},  and the formulae for the cotangent of a sum we get:
\begin{equation}\label{cotom01}
 \cot2\pi\om_{01}=\frac{n_{01}(\sqrt{n_{02}n_{03}}+\scal{v_2,v_3})^2-(\scal{v_3,v_1}\sqrt{n_{02}}+\scal{v_1,v_2}\sqrt{n_{03}})^2-1}{2\sqrt{n_{01}}(\sqrt{n_{02}n_{03}}+\scal{v_2,v_3})}
\end{equation}
and
\begin{equation}\label{cotom02-2om2}
 \cot2\pi(\om_{02}-2\om_2)=\tfrac{(\scal{v_2,v_3-v_2}\sqrt{n_{12}}+\scal{v_1-v_2,v_2}\sqrt{n_{23}})^2-n_{02}(\sqrt{n_{12}n_{23}}+\scal{v_3-v_2,v_1-v_2})^2-1}
 {2(\scal{v_2,v_3-v_2}\sqrt{n_{12}}+\scal{v_1-v_2,v_2}\sqrt{n_{23}})}.
\end{equation}
Rewriting \eqref{cotom01} we get 
\begin{equation}\label{cotom01alt}
 2\sqrt{n_{01}}\cot2\pi\om_{01}=\tfrac{(2n_{01}\scal{v_2,v_3}-2\scal{v_3,v_1}\scal{v_1,v_2})\sqrt{n_{02}n_{03}}+(n_{01}n_{02}n_{03}+n_{01}\scal{v_2,v_3}^2-n_{02}\scal{v_3,v_1}^2
 -n_{03}\scal{v_1,v_2}^2-1)}{\sqrt{n_{02}n_{03}}+\scal{v_2,v_3}}.
\end{equation}
By \eqref{lengths} we have $n_{02}n_{03}\equiv2\pmod4$, hence $\sqrt{n_{02}n_{03}}\notin\QQ(\sqrt{n_{01}},\sqrt{n_{12}},\sqrt{n_{23}})$, therefore
\begin{equation}
 2\sqrt{n_{01}}\cot2\pi\om_{01}\in\QQ(\sqrt{n_{02}n_{03}})\cap\QQ(\sqrt{n_{01}},\sqrt{n_{12}},\sqrt{n_{23}})=\QQ.
\end{equation}
This shows that the numerator and denominator at \eqref{cotom01alt} are $\QQ$-linearly dependent, hence $2\sqrt{n_{01}}\cot2\pi\om_{01}$ is equal to the ratio of the
corresponding coefficients of $\sqrt{n_{02}n_{03}}$, thus
\begin{equation}\label{cotom01int}
 \sqrt{n_{01}}\cot2\pi\om_{01}=n_{01}\scal{v_2,v_3}-\scal{v_3,v_1}\scal{v_1,v_2}\in\ZZ.
\end{equation}
Furthermore, this number is also nonzero, because it is odd, as follows from \eqref{lengths} and \eqref{innerprods}, hence
\begin{equation}\label{lowdegree2}
 \sqrt{n_{01}}\in\QQ(\sqrt{n_{12}},\sqrt{n_{23}}).
\end{equation}
Now we will show that $\QQ(\sqrt{n_{01}})=\QQ(\sqrt{n_{12}})$. If $\QQ(\sqrt{n_{12}},\sqrt{n_{23}})$ is equal to $\QQ$, it is trivial. If it is equal
to a quadratic extension, then $\QQ(\sqrt{n_{12}})=\QQ(\sqrt{n_{23}})\neq\QQ$. If $\sqrt{n_{01}}\in\QQ$, then from \eqref{cotom01int} and \eqref{om02om2om01}
we have that $\cot2\pi(\om_{02}-2\om_2)\in\QQ$. But since $\QQ(\sqrt{n_{12}})=\QQ(\sqrt{n_{23}})\neq\QQ$, the numerator from \eqref{cotom02-2om2} is nonzero rational, while
the denominator is a rational multiple of $\sqrt{n_{12}}$, a contradiction. Hence, in this case, $\QQ(\sqrt{n_{01}})=\QQ(\sqrt{n_{12}})$.

It remains to examine the case where $\QQ(\sqrt{n_{12}},\sqrt{n_{23}})$ is a biquadratic extension. If $\QQ(\sqrt{n_{01}})\neq\QQ(\sqrt{n_{12}})$,
then from \eqref{lowdegree1} and \eqref{lowdegree2} follows that $\QQ(\sqrt{n_{01}})=\QQ(\sqrt{n_{12}n_{23}})$. Recall that by \eqref{om02om2om01} and 
\eqref{cotom01int} we have $\sqrt{n_{01}}\cot2\pi(\om_{02}-2\om_2)\in\ZZ$. However, by \eqref{cotom02-2om2}, the numerator of $\sqrt{n_{01}}\cot2\pi(\om_{02}-2\om_2)$
is a rational linear combination of $1$ and $\sqrt{n_{12}n_{23}}$ and is nonzero, while the denominator is a rational linear combination of $\sqrt{n_{12}}$ and $\sqrt{n_{23}}$,
therefore, they are $\QQ$-linearly independent, and as such their ratio cannot be rational. This contradicts the hypothesis $\QQ(\sqrt{n_{01}})\neq\QQ(\sqrt{n_{12}})$,
hence at all cases we have
\begin{equation}\label{lowdegree3}
 \QQ(\sqrt{n_{01}})=\QQ(\sqrt{n_{12}}).
\end{equation}
Next, \eqref{13} and \eqref{23} yield
\begin{equation}\label{om13om1om23}
 \om_{13}-2\om_1=\om_{23},
\end{equation}
which in turn yields similar formulae to \eqref{cotom01} and \eqref{cotom02-2om2}, where the indices $0$ and $1$ are interchanged with $3$ and $2$, respectively (\emph{notice
that this symmetry is obeyed by the formulae which follow from \eqref{sumeven} and \eqref{sum1mod4}}). Then, similar arguments to those that were used in order
to obtain \eqref{lowdegree3} can be used in order to get
\begin{equation*}
 \QQ(\sqrt{n_{23}})=\QQ(\sqrt{n_{12}}),
\end{equation*}
and thus establish
\begin{equation}\label{lowdegree4}
 \QQ(\sqrt{n_{01}})=\QQ(\sqrt{n_{12}})=\QQ(\sqrt{n_{23}}).
\end{equation}
Therefore, $\sqrt{n_{01}n_{23}}\in\QQ$, hence by Proposition \ref{n03=3} we have $n_{03}=3$, $\Om_1=1/24$, and $\Om_2=1/8$. Then, \eqref{OSOm2} and \eqref{lowdegree4}
yield $1=\cot2\pi\Om_2=d\sqrt{n_{01}}$, for some $d\in\QQ$, thus,
\begin{equation}\label{lowdegree5}
 \QQ(\sqrt{n_{01}})=\QQ(\sqrt{n_{12}})=\QQ(\sqrt{n_{23}})=\QQ.
\end{equation}
Proposition \ref{vectors} gives us once more
\begin{equation}\label{kl}
 v_1=(k+1,k,k),\	v_2=(l,l,l-1),\	v_3=(1,1,1),
\end{equation}
up to an action of $\mc{W}$. In our case, we have $k$ even and $l$ odd from \eqref{lengths}. \eqref{lengths} also gives
 $\sqrt{n_{02}}\notin\QQ(\sqrt{n_{12}},\sqrt{n_{23}})$ and $\sqrt{n_{13}}\notin\QQ(\sqrt{n_{01}},\sqrt{n_{12}})$, hence by \eqref{lowdegree5} and Proposition \ref{tau} we
 get
 \begin{equation}\label{lowdegree6}
  \sqrt{n_{02}}\cot2\pi\om_{02},\sqrt{n_{13}}\cot2\pi\om_{13}\in\QQ.
 \end{equation}
 Now consider $\tau$ to be the nontrivial automorphism of $\QQ(\sqrt{n_{01}},\sqrt{n_{02}},\sqrt{n_{03}})=\QQ(\sqrt{n_{02}},\sqrt{3})$ that fixes $\QQ(\sqrt{3})$
 and $\sig$ be the nontrivial automorphism of $\QQ(\sqrt{n_{03}},\sqrt{n_{13}},\sqrt{n_{23}})=\QQ(\sqrt{n_{13}},\sqrt{3})$ that fixes $\QQ(\sqrt{3})$, i.~e.
 \begin{equation}\label{autos}
  \tau(\sqrt{n_{02}})=-\sqrt{n_{02}},\	\sig(\sqrt{n_{13}})=-\sqrt{n_{13}},\	\tau(\sqrt{3})=\sig(\sqrt{3})=\sqrt{3}.
 \end{equation}
 Finally, let $N_1$ and $N_2$ be the number theoretic norms of the quadratic extensions $\QQ(\sqrt{n_{02}},\sqrt{3})/\QQ(\sqrt{3})$ and $\QQ(\sqrt{n_{13}},\sqrt{3})/\QQ(\sqrt{3})$,
 respectively. By Proposition \ref{tau} we have
 \begin{equation}\label{cotom02}
  \sqrt{n_{02}}\cot2\pi\om_{02}=\frac{-N_1(\cot2\pi\om_0)-1}{2(\sqrt{n_{01}n_{03}}+\scal{v_1,v_3})}
 \end{equation}
 and
 \begin{equation}\label{cotom13}
  \sqrt{n_{13}}\cot2\pi\om_{13}=\frac{-N_2(\cot2\pi\om_3)-1}{2(\sqrt{n_{03}n_{23}}+\scal{-v_3,v_2-v_3})}.
 \end{equation}
Both numerators and denominators of the fractions in \eqref{cotom02} and \eqref{cotom13} belong to $\QQ(\sqrt{3})$, hence by \eqref{lowdegree6}, the left-hand
sides of these equations are also equal to the ratio of the coefficients of $\sqrt{3}$ of the numerator and the denominator, when they are written as $\QQ$-linear
combinations of $1$ and $\sqrt{3}$. We have
\begin{equation}
 -N_1(\cot2\pi\om_0)-1=n_{02}(\sqrt{n_{01}n_{03}}+\scal{v_1,v_3})^2-(\scal{v_2,v_3}\sqrt{n_{01}}+\scal{v_1,v_2}\sqrt{n_{03}})^2-1,
\end{equation}
hence the coefficient of $\sqrt{3}$ is
\begin{equation}
2n_{02}\sqrt{n_{01}}\scal{v_1,v_3}-2\sqrt{n_{01}}\scal{v_1,v_2}\scal{v_2,v_3}, 
\end{equation}
while the coefficient of $\sqrt{3}$ of the denominator in \eqref{cotom02} is just $2\sqrt{n_{01}}$, which yields
\begin{equation}\label{cotom02kl}
 \cot2\pi\om_{02}=\frac{n_{02}\scal{v_1,v_3}-\scal{v_1,v_2}\scal{v_2,v_3}}{\sqrt{n_{02}}}=\frac{2k-l+1}{\sqrt{3l^2-2l+1}},
\end{equation}
by \eqref{kl}. Similarly,
\begin{equation}
 -N_2(\cot2\pi\om_3)-1=n_{13}(\sqrt{n_{03}n_{23}}+\scal{-v_3,v_2-v_3})^2-(\scal{-v_3,v_1-v_3}\sqrt{n_{23}}+\scal{v_2-v_3,v_1-v_3}\sqrt{n_{03}})^2-1,
\end{equation}
hence the coefficient of $\sqrt{3}$ is
\begin{equation}
2n_{13}\sqrt{n_{23}}\scal{-v_3,v_2-v_3}-2\sqrt{n_{23}}\scal{-v_3,v_1-v_3}\scal{v_2-v_3,v_1-v_3}, 
\end{equation}
while the coefficient of $\sqrt{3}$ of the denominator in \eqref{cotom13} is just $2\sqrt{n_{23}}$, which yields
\begin{equation}\label{cotom13kl}
 \cot2\pi\om_{13}=\frac{n_{13}\scal{-v_3,v_2-v_3}-\scal{-v_3,v_1-v_3}\scal{v_2-v_3,v_1-v_3}}{\sqrt{n_{13}}}=\frac{k-2l+2}{\sqrt{3k^2-4k+2}}.
\end{equation}
Equations \eqref{sumeven}, \eqref{cotom02kl}, and \eqref{cotom13kl} yield
\begin{equation}
 \frac{2k-l+1}{\sqrt{3l^2-2l+1}}=\frac{-k+2l-2}{\sqrt{3k^2-4k+2}}.
\end{equation}
Putting $x=k$, $y=-l+1$, the above becomes
\begin{equation}\label{xy}
 \frac{2x+y}{\sqrt{3y^2-4y+2}}=\frac{-x-2y}{\sqrt{3x^2-4x+2}}.
\end{equation}
The rest follows from:
\begin{prop}\label{solution}
 The only integer solution of the equation \eqref{xy} is $x=y=0$.
\end{prop}
\begin{proof}
If $x=y$, then we can easily see that we can only have $x=y=0$, so we may assume that $x\neq y$.
 Square both sides of \eqref{xy} to obtain
 \begin{equation}\label{xy2}
 \frac{(2x+y)^2}{3y^2-4y+2}=\frac{(x+2y)^2}{3x^2-4x+2}.
 \end{equation}
 Both sides are nonnegative and equal to
 \[
 \frac{(2x+y)^2-(x+2y)^2}{(3y^2-4y+2)-(3x^2-4x+2)}=\frac{3(x+y)}{4-3(x+y)},\]
hence we must have $x+y=0$ or $1$. If $x+y=1$, then both sides of \eqref{xy2} must be equal to $3$, hence
\[\frac{(x+1)^2}{3x^2-2x+1}=3,\]
whose only solution is $x=1/2$. Thus, $x+y=0$, hence by \eqref{xy} we have
\[\frac{x}{\sqrt{3x^2+4x+2}}=\frac{x}{\sqrt{3x^2-4x+2}},\]
which yields either $x=0$ or $3x^2+4x+2=3x^2-4x+2$. It is clear, that the only solution is $x=y=0$, as desired.
\end{proof}

Proposition \ref{solution} and \eqref{kl} give
\begin{equation}
 v_1=(1,0,0),\	v_2=(1,1,0),\	v_3=(1,1,1),
\end{equation}
which finally proves  Theorem \ref{converse}.

All such tetrahedra multi-tile $\RR^3$ by the action of the group $\mc{G}$, hence the converse is true in this special case.


\bigskip
\bigskip

\begin{thebibliography}{}

\bibitem{Dirichlet}  

P. G. L. Dirichlet, 
\newblock "Vorlesungen \"Uber Zahlentheorie"
\newblock 4th. ed., {\em Friedrich Vieweg und Sohn, Braunschweig}, 1894. 

\bibitem{Krazer}
A.~Krazer,
\newblock ``Zur Theorie der mehrfachen Gausschen Summen''
\newblock {\em H. Weber Festschrift}, (Leipzig, 1912), s. 181.

\bibitem{McMullen71}
P.~McMullen,
\newblock ``On zonotopes''
\newblock {\em Trans. Amer. Math. Soc.} {\bfseries 159} (1971), 91--109.

\bibitem{McMullen77}
P.~McMullen,
\newblock ``Valuations and Euler-type relations on certain classes of convex polytopes''
\newblock {\em Proc. London Math. Soc.} (3) {\bfseries 35} (1977), no. 1, 113--135.

\bibitem{NZM}
I.~Niven, H.~S.~Zuckerman, and H.~L.~Montgomery,
\newblock ``An Introduction to the Theory of Numbers''
\newblock 5th edition, {\em John Wiley \& Sons, Inc.}, New York (1991).

\bibitem{OS83}
A.~van Oosterom, J.~Strackee,
\newblock ``The solid angle of the plane triangle''
\newblock {\em IEEE Trans. Biomed. Eng.}, Vol. {\bfseries 30}, No.2, 125--126 (1983).


\bibitem{Siegel}  

Carl Ludwig Siegel, 
\newblock "Uber Die Analytische Theorie Der Quadratischen Formen III" 
\newblock{\em Annals of Mathematics, Second Series}, Vol. 38, No. 1, (January 1937), pp. 212-291.



\bibitem{Turaev} 

Vladimir Turaev,
\newblock "Reciprocity for Gauss sums of finite abelian groups"
\newblock {\em Math. Proc. Camb. Phil. Soc.}, (1998) {\bfseries 124}, 205.




\end{thebibliography}
\end{document}